\documentclass[11pt,oneside]{amsart}

\usepackage[a4paper]{geometry}
\usepackage[english]{babel}
\usepackage[utf8]{inputenc}
\usepackage[T1]{fontenc}
\usepackage{lmodern}
\usepackage[all]{xy}
\numberwithin{equation}{section}

\usepackage{amsmath,amsfonts,amssymb,amsthm}
\usepackage{enumerate}
\usepackage{breqn}
\usepackage{tikz}
\usetikzlibrary{trees,arrows,positioning,automata,shadows,fit,shapes}
\usepackage{caption}
\usepackage[numbers,sort]{natbib}
\usepackage{float}
\usepackage{xcolor}
\usepackage[colorlinks=true, pdfstartview=FitV, pagebackref=false]{hyperref}
\usepackage{microtype}

\newcommand{\e}{{\,{\rm e}}}
\newcommand{\var}{{\,{\rm Var}}}
\newcommand{\cov}{{\,{\rm Cov}}}
\newcommand{\PP}{\ensuremath{\mathbb{P}}}
\newcommand{\EE}{\ensuremath{\mathbb{E}}}

\newcommand{\N}{\ensuremath{\mathbb{N}}}

\newcommand{\R}{\ensuremath{\mathbb{R}}}

\newcommand{\cC}{\ensuremath{\mathcal{C}}}

\newcommand{\cE}{\ensuremath{\mathcal{E}}}

\newcommand{\cG}{\ensuremath{\mathcal{G}}}

\newcommand{\cS}{\ensuremath{\mathcal{S}}}
\newcommand{\cL}{\ensuremath{\mathcal{L}}}
\newcommand{\cI}{\ensuremath{\mathcal{I}}}
\newcommand{\cJ}{\ensuremath{\mathcal{J}}}
\newcommand{\cQ}{\ensuremath{\mathcal{Q}}}
\newcommand{\I}{\ensuremath{I}}
\newcommand{\G}{\ensuremath{g}}
\newcommand{\bA}{\ensuremath{\mathbf{A}}}
\newcommand{\bG}{\ensuremath{\mathbf{G}}}

\newcommand{\bX}{\ensuremath{\mathbf{X}}}
\newcommand{\bY}{\ensuremath{\mathbf{Y}}}

\def\ind{{\mathbb{I}}}
\newcommand{\tmix}{\ensuremath{t_{\mathrm{mix}}}}
\newcommand{\trel}{\ensuremath{t_{\mathrm{rel}}}}
\newcommand{\tunif}{\ensuremath{t_{\mathrm{unif}}}}

\newcommand{\given}{\, \big| \,}

\renewcommand{\hat}{\widehat}
\newcommand{\est}{\textsc{est}}
\newcommand{\stopp}{\textsc{stop}}

\newcommand{\rw}{\textsc{rw}}
\newcommand{\lrw}{\textsc{lrw}}
\newcommand{\bt}{\mathbf{t}}

\newtheorem{lemma}{Lemma}

\newtheorem{proposition}[lemma]{Proposition}
\newtheorem{corollary}[lemma]{Corollary}

\theoremstyle{definition}
\newtheorem{remark}{Remark}
\newtheorem{algo}{Algorithm}

\newcommand{\eps}{\varepsilon}

\newcommand{\sX}{\mathcal{X}}
\newcommand{\ip}[2]{\langle #1, #2\rangle_\pi}
\newcommand{\one}{{\bf 1}}

\begin{document}

\title{Estimating graph parameters with random walks}

\author[A. Ben-Hamou]{Anna Ben-Hamou}
 \address{A. Ben-Hamou\hfill\break
Sorbonne Universit\'e, LPSM\\
4, place Jussieu\\ 75005 Paris, France.}
\email{anna.ben-hamou@upmc.fr}

\author[R. Oliveira]{Roberto I. Oliveira}
 \address{R. Oliveira\hfill\break
IMPA\\
Estrada Dona Castorina, 110\\ Rio de Janeiro 22460-320, Brazil.}
\email{rimfo@impa.br}

\author[Y. Peres]{Yuval Peres}
\address{Y.\ Peres\hfill\break
Microsoft Research\\ One Microsoft Way\\ Redmond, WA 98052, USA.}
\email{peres@microsoft.com}
\keywords{graph inference, intersections of random walks}
\subjclass[2010]{60J10, 05C81, 05C85, 62M05}

\begin{abstract}
An algorithm observes the trajectories of random walks over an unknown graph $G$, starting from the same vertex $x$, as well as the degrees along the trajectories. For all finite connected graphs, one can estimate the number of edges $m$ up to a bounded factor in $O\left(\trel^{3/4}\sqrt{m/d}\right)$ steps, where $\trel$ is the relaxation time of the lazy random walk on $G$ and $d$ is the minimum degree in $G$. Alternatively, $m$ can be estimated in $O\left(\tunif +\trel^{5/6}\sqrt{n}\right)$, where $n$ is the number of vertices and $\tunif$ is the uniform mixing time on $G$. The number of vertices $n$  can then be estimated up to a bounded factor in an additional $O\left(\tunif\frac{m}{n}\right)$ steps. Our algorithms are based on counting the number of intersections of random walk paths $X,Y$, \emph{i.e.} the number of pairs $(t,s)$ such that $X_t=Y_s$. This improves on previous estimates which only consider collisions (\emph{i.e.,} times $t$ with $X_t=Y_t$). We also show that the complexity of our algorithms is optimal, even when restricting to graphs with a prescribed relaxation time.  Finally, we show that, given either $m$ or the mixing time of $G$, we can compute the ``other parameter'' with a self-stopping algorithm.
\end{abstract}

\maketitle

\section{Introduction}

What can one learn from the random walk on a graph long before the graph is fully covered? Our motivation is the analysis of large networks that can contain millions (or even billions) of nodes and edges. Direct manipulation or full observations of such huge graphs are typically impractical. Random-walk-based methods, which are local and lightweight, are often used in dealing with this kind of graph (see \citet{DasSarma2013} and the references therein). Our problem, then, is {\em to determine the least number of random walk steps that are needed to compute interesting graph parameters via random walks.} \footnote{This paper is an extended and improved version of the SODA conference proceeding \citep{soda2018paper}.}

We assume our algorithm has black-box access to $K$ random walks of length $t$ on a graph $G$ starting from the same fixed vertex $x$. It then produces an estimate $\widehat{\gamma}_t$ of a parameter $\gamma=\gamma(G)$ of interest, solely by looking at the traces of the random walks and the vertex degrees along the way. The goal is to achieve
\[\forall t\geq t_0\,:\,\PP_x\left(\left|\frac{\widehat{\gamma}_t}{\gamma(G)} -1 \right|\leq \frac{1}{2}\right)\geq 1-\varepsilon,\]
with $t_0$ as small as possible.  In general, the time complexity parameter $t_0$ will depend on the error parameter $\eps$ and on unknown characteristics of the graph. This leads us to consider the possibility of ``self-stopping'' algorithms that decide on their own when to stop exploring $G$.


\subsection{What we do.}

Let us describe our results in more detail, postponing the precise definition of the model to Section~\ref{sec:definitions}. In Section~\ref{sec:intersections-results}, we build on recent results of \citet{peres2017intersection} and \citet{oliveira-peres} to derive bounds on the number of intersections between two independent random walks $X,Y$, i.e. the number of pairs of times $(t,s)$ with $X_t=Y_s$. Using new bounds from~\citep{oliveira-peres}, we show in particular that if $X$ and $Y$ are two independent lazy random walks on $G$, and if $\tau_{\rm{I}}$ denotes the time of the first intersection between $X$ and $Y$, i.e.
$$
\tau_{\rm{I}}=\inf\left\{t\geq 0,\, \{X_0,\dots,X_t\}\cap \{Y_0,\dots,Y_t\}\neq \emptyset\right\}\, ,
$$
then
$$
\max_{x,y\in V}\EE_{x,y}\tau_{\rm{I}} \lesssim \trel^{3/4}\sqrt{m/d}\,
$$
where $m$ is the number of edges in $G$ and $d$ is the minimum degree.

In Section~\ref{sec:estimating-reg}, we focus on the particular case of regular graphs.  Using intersection counts gives us  a simple algorithm for estimating numbers of vertices $n$ of a regular graph $G$ in $O\left(\trel^{3/4}\sqrt{n}\right)$ random walk steps. Moreover, we prove that this algorithm is optimal. More specifically, for any $n$ and $1\lesssim \bt(n)\lesssim n^2$, we construct a graph $G$ with about $n$ vertices and relaxation time about $\bt(n)$. We then show that any \rw\ algorithm that finds the number of vertices of this graph requires at least $\Omega\left(\bt(n)^{3/4}\sqrt{n}\right)$ time steps.

We then consider arbitrary graphs $G$ in Section~\ref{sec:estimating-non-reg}. In section~\ref{subsec:estimating-edges-general}, we show that the number of edges $m$ of $G$ can be estimated in time of order $\left(\trel^{3/4}\sqrt{m/d}\right)\wedge \left(\tunif +\trel^{5/6} \sqrt{n}\right)$, where $\tunif$ is the uniform mixing time on $G$, and we prove in section~\ref{subsec:lower-bound-edges-general} that the bound $\trel^{5/6}\sqrt{n}$ is tight for the estimation of the number of edges on graphs with any prescribed relaxation time. We then show in section~\ref{subsec:vertices-general} that the bound $\tunif^{5/6}\sqrt{n}$, which suffices to estimate the number of edges, may not be sufficient to estimate the number of vertices. However, provided a good estimate for the number of edges is known, the number of vertices follows from estimation of the mean degree, which can be done in times of order $(m/n)\tunif$. Altogether, the number of vertices in general graphs can be estimated with random walks in time of order
$$
\left(\trel^{3/4}\sqrt{m/d}\right)\wedge\left(\trel^{5/6}\sqrt{n}\right)+\tunif\frac{m}{n}, ,
$$
and this is optimal.

Up to this point all algorithms we described are essentially optimal for our model. They are also space-efficient. They just need to store a single real number and maintain a list of visits to each vertex, which is only read or changed during visits. Another desirable trait of our algorithms is that they run in sub-linear time when the mixing time is small (less than $o(n^{3/5})$). This property of (relatively) fast mixing is expected to hold in social networks \cite{leskovec2009} and other large graphs.

However, our algorithms also suffer from a serious drawback: they are not self-stopping. As it turns out, this is unavoidable. We argue in Section~\ref{sec:no-self-stop} that it is not possible to devise a sublinear stopping time at which one can be reasonably sure that our parameters are well estimated. This is true even if our graph is guaranteed to be $3$-regular and have polylog mixing time. We deduce that, while it may be possible to know the size of a graph after sub-linear time, knowing that we already know the size may take much longer.

We complement these results by showing that if either $m$ or the mixing time is known, the other parameter can be estimated with few steps via a self-stopping algorithm. In Section~\ref{sec:algo-edges}, we show how one can use an upper-bound $\tau$ on the mixing time to compute the number of edges via a self-stopping algorithm with time complexity $O\left(\tau^{3/4}\sqrt{m}\log \log m\right)$ (or $O(\tau^{3/4}\sqrt{n}\log \log n)$ steps if $G$ is regular). Section~\ref{sec:algo-mixing} then presents a result for estimating $t_x(\delta)$, the $\ell_2$-mixing time from $x$, with time complexity $O(t_x(\delta/4)^{3/4}\sqrt{m}\log \log t_x(\delta/4))$, assuming a good estimate for the number of edges is available. A corollary is that both the mixing time from $x$ and the number of edges $m$ can be approximated by a self-stopping algorithm with time complexity $O\left(\tau^{3/4}\sqrt{m}\log \log m\right)$, assuming an upper-bound $\tau$ on the uniform mixing time is available.

\subsection{Background}

Our result relates to the a large body of work on inferring graph (or Markov chain) parameters from random walks. We give here a brief overview of these papers, with a focus on results most closely resembling ours.

In some cases, one has to estimate parameters from a single path of the random walk. One possibility is to use return times to the initial vertex to estimate $n$ or $m$, as proposed by \citet{cooper2014estimating} and \citet{benjamini2006}. Other parameters, such as the spectral gap, may be quite challenging to estimate (see \citet{hsu2015mixing} and \citet{levin2016estimating}). In any case, all of these algorithms require time that is at least of the order of the number of vertices, whereas our own algorithms are sublinear in certain cases.

Another line of work, which is closer to ours, is to consider several, say $k$, random walks started from the same vertex $x$. Typically, estimators in this case rely on collisions of random walks at their endpoints. If each random walk has length greater than the mixing time $\tunif$, then the $k$-sample formed by their endpoints is an independent sample  with nearly stationary distribution over the vertex set. In the case where $G$ is regular, the problem comes down to estimating the size of a finite set through independent uniform samples from that set. It is well-known that counting \emph{collisions} and resorting to the \emph{birthday paradox} allow one to estimate $n$ with order $\sqrt{n}$ samples. The time complexity, measured by the total number of random walk steps, is then of order $\tunif\sqrt{n}$ (the same kind of method was also used by \citet{benjamini2007birthday} to estimate the mixing time of regular graphs). If the graph is not regular, the stationary distribution is no longer uniform, and estimation of the support size can be more challenging (see \citep{bunge1993estimating} and \citep{valiant2011estimating} on support size estimation, and \citep{acharya2015optimal} on the related question of testing closeness between distributions). \citet{katzir2014estimating} showed, through a variant of collision counting, that taking $k=O(\sqrt{n}+m/n)$ suffices to estimate $n$ (if one is willing to use more information about the graph, the bound may be improved to $k=O\left(\|\pi\|_2^{-1} +m/n\right)$, where $\|\pi\|_2$ is the Euclidean norm of the stationary distribution $\pi$). \citet{kanade2017large} established a corresponding lower bound for $k$ in this setting. This yields a time complexity of $\tunif(\sqrt{n}+m/n)$. \citet{kanade2017large} asked whether the factor $\tunif$ in those bounds was really necessary or whether more efficient estimators could be designed. Indeed, in those methods, each unit of information already costs $\tunif$ steps. Can we improve the performance by using the information held by the whole trajectories of walks ? We show that this is indeed the case, and that considering intersections of random walks' paths (instead of collisions at their endpoints) gives strictly more information, and leads to optimal time complexity.

\medskip

 Our results are just a first step towards understanding estimation via random walks. It would be interesting to understand what other graph parameters can be computed efficiently in our model. Extensions of our results to oriented graphs and other models of access to the graph (including distributed access as in \cite{DasSarma2013}) would also be worthwhile.

\section{Notation and definitions}\label{sec:definitions}

Let $G=(V,E)$ be a finite connected graph on $n$ vertices and $m$ edges. For $u\in V$, we let $\deg(u)=|\{v\in V, \{u,v\}\in E\}|$ be the degree of $u$.

\subsection*{Random walks and estimators}

Our estimators take as input trajectories of random walks, along with the degrees of visited vertices. However, they do not rely on a particular vertex labeling. To make this more precise, we introduce the  \emph{profile} of a sequence of vertices. For $t\geq 1$ and for a sequence of vertices $\mathbf{u}_t=(u_0,\dots,u_{t-1})\in V^{t}$, let $r(\mathbf{u}_t)$ be the length-$t$ sequence where each vertex is replaced by the index of its first occurrence  in $\mathbf{u}_t$. For instance, the image of the sequence $(g,a,a,c,g,d,a,b,d)$ by $r$ is $(1,2,2,3,1,4,2,5,4)$. Note that $r$ is invariant under vertex-relabeling. The profile $\Phi$ of $\mathbf{u}_t$ is then defined as
\[
\Phi(\mathbf{u}_t)=\Big(r(\mathbf{u}_t),\big(\deg(u_i)\big)_{i=0}^{t-1}\Big)\, .
\]
In other words, for each finite length sequence of vertices $\mathbf{u}_t$, the function $\Phi$ captures the ranks of occurrence and the degrees, and takes values in
\[
\cS = \bigcup_{t\geq 1}\N^{2t} \, .
\]

Now let $x\in V$ be some fixed vertex. An \emph{estimator} is a function $\est:\cS\to\R$, which takes as input the profile of the trajectories of $K$ independent lazy random walks (\lrw) of length $t$, all started at $x$. More precisely, for integers $K, t\geq 1$, let $X^{(1)},\dots, X^{(K)}$ be $K$ independent \lrw\ on $G$ started at $x$, and define $\bX_t^{(i)}=(X_0^{(i)},\dots,X_{t-1}^{(i)})$, the trajectory of $X^{(i)}$ up to time $t-1$, for $i=1,\dots,K$.  Letting $\gamma(G)$ be some parameter of interest (\emph{e.g.} $\gamma(G)=n$ or $\gamma(G)=m$), the goal is to produce a map $\est:\cS\to \R$, returning the value
$$
\hat{\gamma}_{K,t}=\est\left(\Phi\left(\bX_t^{(1)},\dots,\bX_t^{(K)}\right)\right)\, ,
$$
 such that, for all connected graph $G=(V,E)$, for all $x\in V$, for all $t\geq t(\varepsilon,G)$ and $K\geq K(\varepsilon,G)$,
\begin{equation}
\label{eq:def-good-est}
\PP_x\left(\Big| \frac{\hat{\gamma}_{K,t}}{\gamma(G)}-1 \Big| >\frac{1}{2}\right)\leq \varepsilon \, ,
\end{equation}
for $t(\varepsilon,G)\times K(\varepsilon,G)$ as small as possible. The product $t(\varepsilon,G)\times K(\varepsilon,G)$ corresponds to the total number of random walk steps and will often be referred to as the time complexity of the estimator. Let us point out right away that, in our estimation procedures, the critical quantity will be $t(\varepsilon,G)$, the random walks' length, rather than $K(\varepsilon,G)$, the number of random walks, which will simply be chosen according to the desired precision~$\varepsilon$.

\subsection*{Convergence of random walks}

To study the large-time behavior of our estimators, it is natural to take advantage of the convergence of \lrw\ to its stationary distribution $\pi$, given by $\pi(u)=\deg(u)/2m$. Denote by $\tunif$ the uniform mixing time defined as
$$
\tunif = \inf\left\{t\geq 0,\;\max_{x,y\in V} \left|\frac{P^t(x,y)}{\pi(y)}-1\right|\leq \frac{1}{4}\right\}\, ,
$$
Also, letting $1=\lambda_1>\lambda_2\geq\dots\geq \lambda_n\geq 0$ be the eigenvalues of $P$ in decreasing order (the fact that all eigenvalues are non-negative is by laziness of the walk), the relaxation time is defined as
$$
\trel = \left\lceil \frac{1}{1-\lambda_2} \right\rceil \, \cdot
$$

%
%

\subsection*{Self-stopping algorithms}

The time $t(\varepsilon,G)$ above which inequality~\eqref{eq:def-good-est} holds usually depends on unknown parameters of the graph, possibly on $\gamma(G)$ itself. This prompts the search for \emph{self-stopping} algorithms, i.e. algorithms which automatically stop at some random time, according to what has been seen so far. One then needs to control both the error probability for the returned value, and the expected value of the stopping time (see Sections~\ref{sec:no-self-stop}, \ref{sec:algo-edges} and \ref{sec:algo-mixing}).

\section{Intersections of random walks}
\label{sec:intersections-results}

We start by some results on intersections of random walks' trajectories.

For $X$ and $Y$ two independent \lrw\ on a finite connected graph $G=(V,E)$, the number of intersections between $X$ and $Y$ up to time $t-1$ is defined as
$$
\I_t=\sum_{i=0}^{t-1}\sum_{j=0}^{t-1} \ind_{\{X_i=Y_j\}}\, .
$$
For non-regular graphs, a more relevant quantity is the \emph{weighted} number of intersections, defined as
$$
\cI_t=\sum_{i,j=0}^{t-1} \frac{1}{\deg(X_i)}\ind_{\{X_i=Y_j\}}\, .
$$
When $X$ and $Y$ start at $x$ and $y$ respectively, the probability law will be denoted $\PP_{x,y}$ and the corresponding expectation $\EE_{x,y}$. When $x=y$, we just write $\PP_x$ and $\EE_x$. Let $P$ be the transition matrix of $X$ and
$$
\G_t(x,u)=\sum_{i=0}^{t-1} P^i(x,u)
$$
be the expected number of visits to vertex $u$ before time $t$ (also known as the Green's function). We have
\begin{equation}\label{eq:intersection-green}
\EE_x \cI_t =\sum_{u\in V} \frac{\G_t(x,u)^2}{\deg(u)}\, .
\end{equation}
The expected number of intersections is intimately related to return probabilities. Indeed, by reversibility, $\deg(x)\G_t(x,u)=\deg(u)\G_t(u,v)$ and we get
\begin{equation}\label{eq:intersection-return}
\EE_x \cI_t =\sum_{i,j=0}^{t-1} \frac{P^{i+j}(x,x)}{\deg(x)}\, \cdot
\end{equation}
We also define $\cJ_t$ to be the weighted number of intersections counted from the mixing time $\tunif$, i.e.
$$
\cJ_t=\sum_{i,j=\tunif}^{\tunif +t-1}\frac{1}{\deg(X_i)}\ind_{\{X_i=Y_j\}}\, .
$$

\begin{proposition}
\label{prop:intersections}
For all finite connected graph $G=(V,E)$ with $m$ edges, minimum degree $d$ and relaxation time $\trel$, for all $x\in V$,
\begin{equation}\label{eq:prop1-It-exp}
\frac{t^2}{2m}\leq \EE_x \cI_t \leq \frac{t^2}{2m} +\frac{16\trel^{3/2}}{d}\, ,
\end{equation}
and
\begin{equation}\label{eq:prop1-It-var}
\EE_x \cI_t^2 \leq 4 \left(\max_{a\in V} \EE_a\cI_t \right)\EE_x \cI_t\, .
\end{equation}
\end{proposition}

\begin{proposition}
\label{prop:intersections-mix}
For all finite connected graph $G=(V,E)$ with $m$ edges, $n$ vertices and relaxation time $\trel$, for all $x\in V$,
\begin{equation}\label{eq:prop2-Jt-exp}
\left(\frac{3}{4}\right)^2\frac{t^2}{2m}\leq \EE_x \cJ_t \leq \left(\frac{5}{4}\right)^2 \frac{t^2}{2m} \, ,
\end{equation}
and
\begin{equation}\label{eq:prop2-Jt-var}
\EE_x \cJ_t^2 \lesssim \frac{t^2}{m^2}\left(t^2 +n\trel^{5/3}\right)\, .
\end{equation}
\end{proposition}
Here and throughout the paper, for two functions $f,g$, the notation $f(n)\lesssim g(n)$ means that there exists an absolute constant $C>0$ such that $f(n)\leq C g(n)$ for all $n\geq 1$.

Before proving Proposition~\ref{prop:intersections} and~\ref{prop:intersections-mix}, let us state three useful results. The following bound on the Green's function was established by~\citep{oliveira-peres}.

\begin{lemma}[\citep{oliveira-peres}, Lemma 2]
\label{lem:bound-green}
Let $X$ be a \lrw\ on $G$. For all $x\in V$, for all $1\leq t\leq \frac{36m^2}{d}$,
$$
\G_t(x,x)\leq \frac{6\deg(x)\sqrt{t}}{d}\, \cdot
$$
\end{lemma}
By~\citep{oliveira-peres}, Proposition 1, we have
\begin{equation}\label{eq:bound-trel}
\trel \leq \frac{12mn}{d}\, \cdot
\end{equation}
In particular, the bound of Lemma~\ref{lem:bound-green} is valid up to $\trel$.
The following powerful result on the sum of return probabilities was established by \citet{lyons2012sharp}.

\begin{lemma}[\citep{lyons2012sharp}]
\label{lem:lyons-oveis}
For a lazy random walk $X$ on $G$, for all $t\geq 0$,
$$
\sum_{u\in V} P^t(u,u)\leq 1+\frac{13n}{(t+1)^{1/3}}\, \cdot
$$
\end{lemma}

Finally, we also need the following lemma.

\begin{lemma}
\label{lem:sumfint2}
For any $f\in \R^\sX$, if $P$ is reversible, irreducible and has non-negative spectrum, then
$$
\sum_{s=0}^{+\infty} (s+1)\left(\ip{f}{P^{s}f} - \ip{f}{\one}^2\right)\leq \frac{\trel}{(1-1/\e)^2}\sum_{s=0}^{\trel -1} [\ip{f}{P^{s}f} - \ip{f}{\one}^2]\, .
$$
\end{lemma}

\begin{proof}[Proof of Lemma \ref{lem:sumfint2}]
Partitioning $\N$ in blocks of length $\trel$, we may write
$$
\sum_{s=0}^{+\infty}(s+1)\left(\ip{f}{P^{s}f} - \ip{f}{\one}^2\right) = \sum_{k=0}^{+\infty}\sum_{s=0}^{\trel-1}(\trel k +s+1)\left(\ip{f}{P^{\trel k+s}f} - \ip{f}{\one}^2\right)\, .
$$
The terms in the above sums can be written in the form:
\[\ip{f}{P^{r}f} - \ip{f}{\one}^2 = \sum_{j=2}^n\,\lambda_j^r\,\ip{f}{\Psi_j}^2\]
where $\lambda_1=1>\lambda_2\geq \lambda_2\geq \dots \leq \lambda_n\geq 0$ are the eigenvalues of $P$ and $(\Psi_1=\one,\Psi_2,\dots,\Psi_n)$ is an orthonormal basis of eigenvectors for the inner product $\ip{\cdot}{\cdot}$. By definition of $\trel$, we have $\lambda_j^{\trel k} \leq \e^{-k}$ for all $j\geq 2$. Therefore,
\[\ip{f}{P^{\trel k+s}f} - \ip{f}{\one}^2\leq \e^{-k}\,[\ip{f}{P^{s}f} - \ip{f}{\one}^2].\]
Summing the bounds, we obtain
\begin{eqnarray*}
\sum_{s=0}^{+\infty}(s+1)\left( \ip{f}{P^{s}f} - \ip{f}{\one}^2\right)&\leq &
\sum_{k=0}^{+\infty}  \sum_{s=0}^{\trel -1}(\trel k +s+1) \e^{-k}\left( \ip{f}{P^{s}f} - \ip{f}{\one}^2\right)\\
&\leq &
\sum_{k=0}^{+\infty} \trel (k +1) \e^{-k} \sum_{s=0}^{\trel -1}\left( \ip{f}{P^{s}f} - \ip{f}{\one}^2\right)\\
&\leq & \frac{\trel}{(1-1/\e)^2} \sum_{s=0}^{\trel -1} \,[\ip{f}{P^{s}f} - \ip{f}{\one}^2]\, .
\end{eqnarray*}
\end{proof}

\begin{proof}[Proof of Proposition \ref{prop:intersections}]
By~\eqref{eq:intersection-return}, we have
$$
\EE_x \cI_t =\sum_{i,j=0}^{t-1} \frac{P^{i+j}(x,x)}{\deg(x)}=\frac{t^2}{m}+\frac{1}{2m}\sum_{i,j=0}^{t-1} \left(\frac{P^{i+j}(x,x)}{\pi(x)}-1\right)\, .
$$
All summands in the right-hand side are non-negative (this can be seen, for instance, by the spectral decomposition $P^r(x,x)=\pi(x)+\sum_{j=2}^n \lambda_j^r\Psi_j(x)^2\pi(x)$ and by non-negativity of the eigenvalues). Moreover, by Lemma~\ref{lem:sumfint2} applied to the function $f=\frac{\ind_{\{\cdot=x\}}}{\pi(x)}$,
\begin{eqnarray}\label{eq:centered-return}
\sum_{i,j=0}^{t-1} \left(\frac{P^{i+j}(x,x)}{\pi(x)}-1\right)&\leq &\sum_{s=0}^{+\infty}(s+1)\left(\frac{P^{s}(x,x)}{\pi(x)}-1\right)\nonumber\\
&\leq & \frac{\trel}{(1-1/\e)^2}\sum_{s=0}^{\trel -1}\left(\frac{P^{s}(x,x)}{\pi(x)}-1\right)\\
&\leq & \frac{\trel}{(1-1/\e)^2} \frac{\G_\trel(x,x)}{\pi(x)}\, \cdot \nonumber
\end{eqnarray}
Resorting to Lemma~\ref{lem:bound-green}, we obtain
$$
\EE_x \cI_t \leq \frac{t^2}{2m} +\frac{6\trel^{3/2}}{(1-1/\e)^2 d} \leq \frac{t^2}{2m} + \frac{16\trel^{3/2}}{d}\, ,
$$
concluding the proof of the first moment bounds. Moving on to the second moment, we have
\begin{eqnarray*}
\EE_x \cI_t^2
&=& \sum_{u,v}\frac{1}{\deg(u)\deg(v)}\left(\sum_{i,k=0}^{t-1} \PP_x(X_i=u,X_k=v)\right)^2\\
&\leq & \sum_{u,v}\frac{1}{\deg(u)\deg(v)}\left(\G_t(x,u)\G_t(u,v)+\G_t(x,v)\G_t(v,u)\right)^2\\
&\leq & 4 \sum_{u,v}\frac{\G_t(x,u)^2\G_t(u,v)^2}{\deg(u)\deg(v)} \\
&=& 4\sum_{u} \frac{\G_t(x,u)^2}{\deg(u)}\EE_u\cI_t\, \leq \, 4 \left(\max_{u\in V} \EE_u \cI_t\right)\EE_x \cI_t\, ,
\end{eqnarray*}
and~\eqref{eq:prop1-It-var} follows from the upper-bound in~\eqref{eq:prop1-It-exp}.
\end{proof}

\begin{proof}[Proof of Proposition~\ref{prop:intersections-mix}]
The bounds on the expectation of $\cJ_t$ are straightforward. Indeed
$$
\EE_x\cJ_t = \sum_{y,z} P^\tunif(x,y) P^\tunif(x,z)\EE_{y,z}\cI_t\, ,
$$
so that, by definition of $\tunif$ and the fact that $\sum_{y,z} \pi(y)\pi(z)\EE_{y,z}\cI_t =t^2/2m$,
$$
\left(\frac{3}{4}\right)^2 \frac{t^2}{2m}\leq \EE_x\cJ_t \leq \left(\frac{5}{4}\right)^2 \frac{t^2}{2m}\, \cdot
$$
Moving on to~\eqref{eq:prop2-Jt-var}, again by definition of $\tunif$, we have
\begin{eqnarray*}
\EE_x\cJ_t^2 &\lesssim & \sum_{y,z}\pi(y)\pi(z)\EE_{y,z}\cI_t^2\\
&\lesssim & \sum_{y,z}\sum_{u,v}\frac{\pi(y)\pi(z)}{\deg(u)\deg(v)}\sum_{i,j,k,\ell}\PP_y(X_i=u,X_k=v)\PP_z(Y_j=u,Y_\ell=v)\\
&\lesssim & \sum_{u,v}\frac{1}{\deg(u)\deg(v)}\left(\sum_{y}\pi(y)\sum_{i,k}\PP_y(X_i=u,X_k=v)\right)^2\\
&\lesssim & \sum_{u,v}\frac{1}{\deg(u)\deg(v)}\left(\sum_{y}\pi(y)\G_t(y,u)\G_t(u,v)\right)^2\, .
\end{eqnarray*}
Using that $\sum_y\pi(y)\G_t(y,u)=\sum_y \pi(u)\G_t(u,y)=t\pi(u)$, we have
$$
\EE_x\cJ_t^2  \lesssim \frac{t^2}{m^2}\sum_{u,v}\frac{\deg(u)}{\deg(v)}\G_t(u,v)^2=\frac{t^2}{m^2}\sum_{i,j=0}^{t-1}\sum_{u} P^{i+j}(u,u)\, ,
$$
where the last equality comes from reversibility. Now, by inequality~\eqref{eq:centered-return},
\begin{eqnarray*}
\sum_{i,j=0}^{t-1}\sum_{u} P^{i+j}(u,u)&=& t^2 +\sum_{u}\pi(u)\sum_{i,j=0}^{t-1} \left(\frac{P^{i+j}(u,u)}{\pi(u)}-1\right)\\
&\leq & t^2 + \frac{\trel}{(1-1/\e)^2}\sum_{s=0}^{\trel -1}\left(\sum_{u}P^s(u,u)-1\right) \, .
\end{eqnarray*}
Finally, resorting to Lemma~\ref{lem:lyons-oveis}, we obtain
$$
\EE_x\cJ_t^2 \lesssim \frac{t^2}{m^2}\left(t^2+ n\trel^{5/3}\right)\, ,
$$
concluding the proof of Proposition~\ref{prop:intersections-mix}.
\end{proof}

\begin{remark}
Proposition~\ref{prop:intersections} entails bounds on $\EE_{x,y}\cI_t$. Indeed, for $x\neq y$, we may use the bound
$$
\left|\frac{P^{t}(x,y)}{\pi(y)} - 1\right|\leq \sqrt{\frac{P^{t}(x,x)}{\pi(x)} - 1}\,\sqrt{\frac{P^{t}(y,y)}{\pi(y)} - 1}\, ,
$$
which follows, for instance, from Cauchy-Schwarz Inequality in the spectral decomposition $P^t(x,y)=\pi(y)\left(1+\sum_{j=2}^n \lambda_j^t \Psi_j(x)\Psi_j(y)\right)$. This entails
$$
\left|\EE_{x,y}{\cI_t} - \frac{t^2}{2m}\right| \leq \sqrt{\EE_{x,x}{\cI_t} - \frac{t^2}{2m}}\,\sqrt{\EE_{y,y}{\cI_t} - \frac{t^2}{2m}}\, .
$$
Moreover, one may check easily that $\max_{x,y}\EE_{x,y}\cI_t^2\leq \max_{x}\cI_t^2$. From those bounds, one may derive the following new bound  on the first intersection time: for $t\gtrsim \trel^{3/4}\sqrt{m/d}$, by the second-moment method, $\PP_{x,y}(\cI_t>0)\geq 1/8$. Since this holds uniformly in $x$ and $y$, one may perform independent experiments to conclude that
$$
\max_{x,y}\EE_{x,y}\tau_{\textrm{I}} \lesssim \trel^{3/4}\sqrt{m/d}\, .
$$
\end{remark}

\section{Estimating the number of vertices on regular graphs}
\label{sec:estimating-reg}

\subsection{A simple estimator for the number of vertices}

Specifying to regular graphs with degree $d\geq 1$ and considering the unweighted number of intersections $\I_t$, Proposition~\ref{prop:intersections} entails
$$
\frac{t^2}{n} \leq \EE_x\I_t \leq \frac{t^2}{n} + 16\trel^{3/2}\, ,
$$
and
$$
\EE_x\I_t^2 \lesssim \left(\frac{t^2}{n} +\trel^{3/2}\right)^2\, .
$$
This suggests the following simple estimator for the number of vertices in a regular graph: consider $2K$ independent lazy random walks $X^{(1)}, Y^{(1)},\dots, X^{(K)}, Y^{(K)}$ all started at the same vertex $x\in V$. For each $k$ between $1$ and $K$, let $\I_t^{(k)}$ be the number of intersections of $X^{(k)}$ and $Y^{(k)}$ between $0$ and $t-1$, and define
\begin{eqnarray}
\label{eq:hat-n}
\widehat{n}_t&=&\frac{t^2}{\frac{1}{K}\sum_{k=1}^K\I_t^{(k)}}\, \cdot
\end{eqnarray}

For  $t\geq 2\sqrt{6}\trel^{3/4}\sqrt{n}$, we have $\frac{t^2}{n}\leq \EE_x\I_t \leq \frac{5t^2}{3n}$ and $\var_x \I_t\lesssim t^4/n^2$. Hence, by Chebyshev's Inequality
\[
\PP_x\left(\Big| \frac{\widehat{n}_t}{n} -1\Big| >\frac{1}{2}\right) \leq  \PP_x\left(\Big|\frac{1}{K}\sum_{k=1}^K\I_t^{(k)}-\EE_x \cI_t\Big| >\frac{t^2}{3n}\right)=O\left(\frac{1}{K}\right)\, .
\]

\subsection{Lower bounds for regular graphs}\label{subsec:lower-bound-reg}

The case of the cycle on $n$ vertices gives an example where the bound $\trel^{3/4}\sqrt{n}$ is tight. Indeed, in this case, $\trel\asymp n^2$, and thus $\trel^{3/4}\sqrt{n}\asymp n^2$. And any procedure based on random walks requires at least order $n^2$ steps to distinguish between a cycle of size $n$ and a cycle of size $2n$.

This section is devoted to a stronger version of tightness. Namely, we exhibit graphs achieving the bound $\trel^{3/4}\sqrt{n}$ for any $n$, \emph{and} for any relaxation time $\trel$.

\begin{proposition}
\label{prop:lower-bound}
There exist absolute constants $\delta,\Lambda>0$ such that the following holds. For all integers $n\geq \Lambda$ and $\bt(n)$ with $\Lambda\leq \bt(n)\leq \Lambda n^2$, for all map $\est:\cS\to \R$, there exists a $3$-regular graph $G=(V,E)$ such that
\begin{itemize}
\item $|V|\in \left[n, 14n\right]$;
\item $\trel  \leq \bt(n)$;
\item for more than $9/10^{\mbox{th}}$ of the vertices $x\in V$, for all $t,K\geq 1$ with $tK\leq \delta \bt(n)^{3/4} \sqrt{n}$,
$$
\PP_x\left(\left| \frac{\hat{n}_t}{n}-1\right|>\frac{1}{2}\right)\geq \frac{1}{4}\, ,
$$
where $\hat{n}_t=\est\left(\Phi\left(\bX_t^{(1)},\dots,\bX_t^{(K)}\right)\right)$.
\end{itemize}
\end{proposition}

Before proving Proposition~\ref{prop:lower-bound}, we first establish the following lemma.

\begin{lemma}
\label{lem:lower-bound-random-graph}
For $k\geq 2$ even, let $G_k$ be a uniform random $3$-regular graph on $k$ vertices. Then
\begin{enumerate}
\item The probability that $G_k$ is connected tends to $1$ as $k\to \infty$;
\item The relaxation time $\trel(G_k)$ tends to $(1-2\sqrt{2}/3)^{-1}$ in probability;
\item For $k$ large enough, for all $x\in V(G_k)$, letting $(X_s)_{s\geq 0}$ be the concatenation of independent \rw s of length $t\geq 1$ on $G_k$ started at $x$ (i.e. $(X_s)_{s\geq 0}=(\bX_t^{(1)},\bX_t^{(2)},\dots)$), we have, as soon as $s\leq \sqrt{k}/20$,
$$
\PP_x\left(\bG_s \mbox{ is a tree } \right)\geq \frac{93}{95}\, ,
$$
where $\bG_s$ is the subgraph induced by the edges visited by $(X_0,\dots,X_{s-1})$.
\end{enumerate}

\end{lemma}
\begin{proof}[Proof of Lemma~\ref{lem:lower-bound-random-graph}]
The first item is a well-known fact, valid for random graphs with given degrees, as soon as the minimum degree is larger or equal to $3$. The second item is by Friedman's Theorem~\citep{friedman2008proof}, which states that a random $d$-regular graph is with high probability weakly Ramanujan, i.e. its relaxation time is asymptotic to $(1-2\sqrt{d-1}/d)^{-1}$. Now, to establish the third item, we use a common method to generate a uniform $3$-regular random graph, known as the \emph{configuration model} (see \citep{configuration1}). One initially considers $k$ isolated vertices, each vertex $v$ being endowed with $3$ \emph{half-edges} $(v,1)$, $(v,2)$, $(v,3)$. A random matching on half-edges is then chosen uniformly, and each pair of matched half-edges is interpreted as an edge between the corresponding vertices. The probability that this creates a  simple graph tends to $\e^{-2}$ (see for instance~\citep{janson2009probability}), and, conditionally on being simple, the graph is uniformly distributed over simple $3$-regular graphs. One nice feature of this model is that it allows to generate sequentially and simultaneously the graph and the random walks, as follows. Initially, all half-edges are unpaired and $X_0=x$. Then, at each step $s\geq 1$,
\begin{itemize}
\item either $s$ is a multiple of $t$ and we set $X_{s}=x$ (hereby starting a new walk),
\item or $s$ is not a multiple of $t$ and we then choose with probability $1/3$ a half-edge $(X_{s-1},*)$ attached to $X_{s-1}$. If $(X_{s-1},*)$ has already been paired to some half-edge $(v,*)$, we let $X_{s}=v$. Otherwise, we choose uniformly at random an unpaired half-edge $(u,*)$, match $(X_{s-1},*)$ and $(u,*)$, and let $X_{s}=u$.
\end{itemize}
Observe that the edges spanned by $(X_s)$ form a tree up to the first time $s$ when $(X_{s-1},*)$ is unpaired but is then matched to a half-edge attached to a visited vertex (creating a cycle in the induced graph). The probability that this event first occurs at time $s$ is smaller than $\frac{3s}{3k-3s}$ (by time $s-1$, we have exposed at most $3s$ half-edges). Hence, the (annealed) probability that this event occurs before time $s$ is smaller than $\frac{3s^2}{3k-3s}$. For $s=\sqrt{k}/20$, this probability is smaller than $1/380$. For $k$ large enough, the probability for the configuration model to yield a simple graph is larger than $1/8$, hence, on $G_k$, we have $\PP_x\left(\bG_s \mbox{ is a tree } \right)\geq 1- 8/380 =93/95$.
\end{proof}

Lemma~\ref{lem:lower-bound-random-graph} entails the following: there exists $k_0\geq 1$ such that for all even $k\geq k_0$, there exist connected $3$-regular graphs $\cE_k$ and $\cE_{4k}$ on $k$ and $4k$ vertices respectively, satisfying
\begin{equation}\label{eq:trel}
\max\{\trel(\cE_k), \trel(\cE_{4k})\}\leq 18\, ,
\end{equation}
and, for more than $9/10^{\mbox{th}}$ of the pairs of vertices $(x,y)\in V(\cE_k)\times V(\cE_{4k})$, there is a coupling of $(X_s)$ and $(Y_s)$, where $(X_s)$ (resp. $(Y_s)$) is the concatenation of independent \rw s of length $t$ on $\cE_k$ (resp. $\cE_{4k}$) started at $x$ (resp. $y$) such that, if $s\leq \sqrt{k}/20$,
\begin{equation}\label{eq:coupling}
\PP_{x,y}\left(\Phi(X_0^s)=\Phi(Y_0^s) \right)\geq  \frac{3}{4}\, \cdot
\end{equation}
Indeed, on uniform $3$-regular random graphs $G_k$ and $G_{4k}$, the two processes $(X_s)$ and $(Y_s)$ can be successfully coupled up to the first time $s$ when $\bG_s$ is not a tree. By Lemma~\ref{lem:lower-bound-random-graph}, this has probability less than $2/95$ for $s\leq \sqrt{k}/20$. Letting $\mathbf{P}_{x,y}$ denote the (quenched) probability associated with the coupled random walks on $G_k$ and $G_{4k}$, by Markov's Inequality (applied twice),
\begin{eqnarray*}
\PP\left(\left|\left\{ (x,y),\, \mathbf{P}_{x,y}\left(\Phi(X_0^s)=\Phi(Y_0^s) \right)<\frac{3}{4}\right\}\right|>\frac{1}{10}(k\times 4k)\right)&\leq & \frac{80}{95}\, \cdot
\end{eqnarray*}
Hence we can find graphs $\cE_k$ and $\cE_{4k}$ satisfying~\eqref{eq:coupling}.

\begin{proof}[Proof of Proposition~\ref{prop:lower-bound}]
For some constant $\Lambda>0$ to be specified later, let $n\geq \Lambda$ and $\Lambda \leq \bt(n)\leq \Lambda n^2$, and define
$$
\ell =4\left\lfloor \frac{1}{4} \sqrt{\frac{\bt(n)}{\Lambda}} \right\rfloor +1 \qquad \mbox{and} \qquad k=\left\lceil \frac{2n}{3\ell -1} \right\rceil \, .
$$
Now let $\cG_{k,\ell}$ and $\cG_{4k,\ell}$ be constructed as follows:
\begin{enumerate}
\item take two $3$-regular graphs $\cE_k$ and $\cE_{4k}$ satisfying~\eqref{eq:coupling} (by our assumptions on $n$ and $\bt(n)$, the constant $\Lambda$ can be chosen large enough so that $k\geq k_0$);
\item in each graph, in place of each edge, put a path of length $\ell$.;
\item to make those graphs $3$-regular, add edges between pairs of interior vertices at distance $2$ on the same path (this is possible because $\ell-1$ is a multiple of $4$).
\end{enumerate}
See Figure~\ref{fig:lower-bound}.
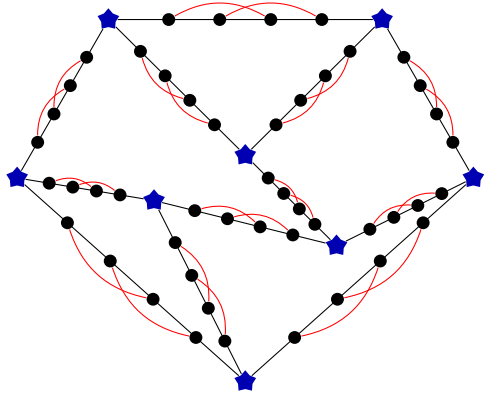
\begin{figure}[ht]
\centering
\begin{tikzpicture}[scale=0.6,sibling distance=8em,
  every node/.style = {fill,circle,minimum size=1.7mm,inner sep=0}]
\node[blue!70!black,star,minimum size=3mm] at (-2,1) (A) {};
\node[blue!70!black,star,minimum size=3mm] at (0,2) (B) {};
\node[blue!70!black,star,minimum size=3mm] at (2,0) (C) {};
\node[blue!70!black,star,minimum size=3mm] at (-5,1.5) (D) {};
\node[blue!70!black,star,minimum size=3mm] at (-3,5) (E) {};
\node[blue!70!black,star,minimum size=3mm] at (3,5) (F) {};
\node[blue!70!black,star,minimum size=3mm] at (5,1.5) (G) {};
\node[blue!70!black,star,minimum size=3mm] at (0,-3) (H) {};
\draw (A)--(C) node[pos=0.2] (ac1) {} node[pos=0.4] (ac2) {}  node[pos=0.6] (ac3) {} node[pos=0.8] (ac4) {};
\draw (A)--(H) node[pos=0.2] (ah1) {} node[pos=0.4] (ah2) {}  node[pos=0.6] (ah3) {} node[pos=0.8] (ah4) {};
\draw (B)--(C) node[pos=0.2] (bc1) {} node[pos=0.4] (bc2) {}  node[pos=0.6] (bc3) {} node[pos=0.8] (bc4) {};
\draw (A)--(D) node[pos=0.2] (ad1) {} node[pos=0.4] (ad2) {}  node[pos=0.6] (ad3) {} node[pos=0.8] (ad4) {};
\draw (D)--(H) node[pos=0.2] (dh1) {} node[pos=0.4] (dh2) {}  node[pos=0.6] (dh3) {} node[pos=0.8] (dh4) {};
\draw (H)--(G) node[pos=0.2] (hg1) {} node[pos=0.4] (hg2) {}  node[pos=0.6] (hg3) {} node[pos=0.8] (hg4) {};
\draw (C)--(G) node[pos=0.2] (cg1) {} node[pos=0.4] (cg2) {}  node[pos=0.6] (cg3) {} node[pos=0.8] (cg4) {};
\draw (F)--(G) node[pos=0.2] (fg1) {} node[pos=0.4] (fg2) {}  node[pos=0.6] (fg3) {} node[pos=0.8] (fg4) {};
\draw (D)--(E) node[pos=0.2] (de1) {} node[pos=0.4] (de2) {}  node[pos=0.6] (de3) {} node[pos=0.8] (de4) {};
\draw (E)--(F) node[pos=0.2] (ef1) {} node[pos=0.4] (ef2) {}  node[pos=0.6] (ef3) {} node[pos=0.8] (ef4) {};
\draw (B)--(E) node[pos=0.2] (be1) {} node[pos=0.4] (be2) {}  node[pos=0.6] (be3) {} node[pos=0.8] (be4) {};
\draw (B)--(F) node[pos=0.2] (bf1) {} node[pos=0.4] (bf2) {}  node[pos=0.6] (bf3) {} node[pos=0.8] (bf4) {};

\draw[red] (ac1) to[bend left] (ac3);
\draw[red] (ac2) to[bend left] (ac4);

\draw[red] (ah1) to[bend left] (ah3);
\draw[red] (ah2) to[bend left] (ah4);

\draw[red] (bc1) to[bend left] (bc3);
\draw[red] (bc2) to[bend left] (bc4);

\draw[red] (ad1) to[bend right] (ad3);
\draw[red] (ad2) to[bend right] (ad4);

\draw[red] (dh1) to[bend right] (dh3);
\draw[red] (dh2) to[bend right] (dh4);

\draw[red] (hg1) to[bend right] (hg3);
\draw[red] (hg2) to[bend right] (hg4);

\draw[red] (cg1) to[bend left] (cg3);
\draw[red] (cg2) to[bend left] (cg4);

\draw[red] (fg1) to[bend left] (fg3);
\draw[red] (fg2) to[bend left] (fg4);

\draw[red] (de1) to[bend left] (de3);
\draw[red] (de2) to[bend left] (de4);

\draw[red] (ef1) to[bend left] (ef3);
\draw[red] (ef2) to[bend left] (ef4);

\draw[red] (be1) to[bend left] (be3);
\draw[red] (be2) to[bend left] (be4);

\draw[red] (bf1) to[bend right] (bf3);
\draw[red] (bf2) to[bend right] (bf4);
\end{tikzpicture}
\captionsetup{margin=3cm}
\caption{The graph $\cG_{k,\ell}$ ($k=8$, $\ell=5$).
The blue star-shaped vertices are the original vertices of~$\cE_k$.}
\label{fig:lower-bound}
 \end{figure}

Note that, using $\ell\leq n+1$,
$$
n\leq |V(\cG_{k,\ell})|=\frac{k}{2}(3\ell-1)\leq \frac{7n}{2}\, ,
$$
and similarly $4n\leq |V(\cG_{4k,m})|\leq 14n$. Moreover, choosing $\Lambda$ large enough, we have
$$
\max\left\{\trel(\cG_{k,\ell}), \trel(\cG_{k,\ell})\right\}\leq \frac{\Lambda}{4} \ell^2 \, .
$$
This can be seen by conductance arguments (the bottleneck ratio of $\cE_k$ is bounded away from $0$ by expansion, entailing that the one of $\cG_{k,\ell}$ is up to constant factors larger than $1/\ell$, and by Cheeger's Inequality, the relaxtion time is smaller than $\ell^2$ up to constant factors). By definition of $\ell$ and the fact that $\Lambda\leq \bt(n)$,
$$
\max\left\{\trel(\cG_{k,\ell}), \trel(\cG_{4k,\ell})\right\}\leq  \frac{\Lambda}{4}\left(\sqrt{\frac{\bt(n)}{\Lambda}}+1\right)^2 \leq  \frac{\Lambda}{4}\left(2\sqrt{\frac{\bt(n)}{\Lambda}}\right)^2\leq \bt(n)\, .
$$
Combining equation~\eqref{eq:coupling} and the $\ell^2$-slow down induced by paths, we obtain that for $9/10$ of the starting points $(x,y)\in V(G_{k,\ell})\times V(G_{4k,\ell})$, there is a coupling of random walks such that, letting
$$
\bA_t=\left\{\Phi\left(\bX_t^{(1)},\dots,\bX_t^{(K)}\right)=\Phi\left(\bY_t^{(1)},\dots,\bY_t^{(K)}\right)\right\},
$$
we have
\begin{equation}
\label{eq:coupling2}
\PP_{x,y}\left(\bA_t\right) \geq  \frac{3}{4}\, ,\quad\mbox{with}\quad  Kt=\delta \ell^2\sqrt{k}\, ,
\end{equation}
for some $\delta>0$ small enough. Let $\est:\cS\to \N$ be an estimator and let $\hat{n}_t(X)= \est\left(\Phi\left(\bX_t^{(1)},\dots,\bX_t^{(K)}\right)\right)$ and $\hat{n}_t(Y)= \est\left(\Phi\left(\bY_t^{(1)},\dots,\bY_t^{(K)}\right)\right)$. Define
\[
B_t^{X}=\left\{\frac{1}{2}\leq \frac{\hat{n}_t(X)}{n}\leq \frac{3}{2}\right\}\, ,\quad\mbox{and }\quad
B_t^{Y}=\left\{\frac{1}{2}\leq \frac{\hat{n}_t(Y)}{4n}\leq \frac{3}{2}\right\}\, .
\]
Assume that it holds simultaneously that $\PP_{x}\left(B_t^{X}\right)\geq 3/4$ and $\PP_{y}\left(B_t^{Y}\right)\geq 3/4$. Then, by~\eqref{eq:coupling2},
\begin{displaymath}
\PP_{x,y}\left(B_t^{X}\given \bA_t\right)=\frac{\PP_{x,y}\left(B_t^{X}\cap \bA_t\right)}{\PP_{x,y}(\bA_t)}\geq 1-\frac{1-\PP_{x}\left(B_t^{X}\right)}{\PP_{x,y}(\bA_t)}\geq \frac{2}{3}\, ,
\end{displaymath}
and similarly, $\PP_{x,y}\left(B_t^{Y}\given \bA_t\right) \geq \frac{2}{3}$, so that $\PP_{x,y}\left(B_t^{X}\cap B_t^{Y}\given \bA_t\right) \geq  \frac{1}{3}$. However, on the event $\bA_t$, the events $B_t^{X}$ and $B_t^{Y}$ can not occur simultaneously, implying a contradiction. We either have $\PP_{x}\left(\big| \frac{\hat{n}_t(X)}{n}-1\big|>\frac{1}{2}\right)\geq \frac{1}{4}$ or $\PP_{y}\left(\big| \frac{\hat{n}_t(Y)}{4n}-1\big|>\frac{1}{2}\right)\geq \frac{1}{4}$. The proof is then concluded by noticing that
\begin{eqnarray*}
\ell^2\sqrt{k}&\gtrsim & \bt(n)^{3/4}\sqrt{n}\, .
\end{eqnarray*}
\end{proof}

\section{Computing parameters of general graphs}
\label{sec:estimating-non-reg}

\subsection{A simple estimator for the number of edges}
\label{subsec:estimating-edges-general}

In the non-regular case, Proposition~\ref{prop:intersections} suggests the following simple estimator for the number of edges, namely:
\begin{eqnarray}
\label{eq:hat-m}
\widehat{m}_t&=&\frac{t^2}{\frac{2}{K}\sum_{k=1}^K\cI_t^{(k)}}\,  ,
\end{eqnarray}
where $\left\{\cI_t^{(k)}\right\}_{k=1}^K$ are independent copies of $\cI_t$, the weighted number of intersections between to independent random walks started at some $x\in V$. For $t\geq 4\sqrt{3}\trel^{3/4}\sqrt{m/d}$, we have $\frac{t^2}{2m}\leq \EE_x\cI_t \leq \frac{5t^2}{6m}$ and $\var_x \cI_t\lesssim t^4/m^2$. Hence, by Chebyshev's Inequality
\[
\PP_x\left(\Big| \frac{\widehat{m}_t}{m} -1\Big| >\frac{1}{2}\right)= \PP_x\left(\Big| \frac{1}{K}\sum_{k=1}^K \cI_t^{(k)} -\EE_x\cI_t \Big| >\frac{t^2}{6m}\right) = O\left(\frac{1}{K}\right)\, .
\]

Alternatively, considering the other estimator
\begin{eqnarray}
\label{eq:hat-m2}
\widetilde{m}_t&=&\frac{t^2}{\frac{2}{K}\sum_{k=1}^K\cJ_t^{(k)}}\,  ,
\end{eqnarray}
where $\left\{\cJ_t^{(k)}\right\}_{k=1}^K$ are independent copies of $\cJ_t$, we obtain, by Proposition~\ref{prop:intersections-mix}, that for $t\gtrsim \trel^{5/6}\sqrt{n}$,
\[
\PP_x\left(\Big| \frac{\widetilde{m}_t}{m} -1\Big| >\frac{1}{2}\right) = O\left(\frac{1}{K}\right)\, .
\]
Since intersections are counted from the uniform mixing time, the total time complexity of $\widetilde{m}_t$ to reach error probability $\varepsilon$ is $O\left(\varepsilon^{-1}(\tunif +\trel^{5/6}\sqrt{n})\right)$.

\subsection{Lower bounds for general graphs}
\label{subsec:lower-bound-edges-general}

The bound $\trel^{5/6}\sqrt{n}$ is achieved on a graph known as the barbell, formed by two cliques of size $n$ joined by a path of length $n$. Indeed, the relaxation time of this graph has order $n^3$, so that $\trel^{5/6}\sqrt{n} \asymp n^3$, and any procedure based on random walks needs time $n^3$ to correctly estimate $n$, since this is the time needed by a random walk to go from one clique to the other.

As in Section~\ref{subsec:lower-bound-reg}, we now exhibit graphs achieving the bound $\trel^{5/6}\sqrt{n}$ for any $n$ and any relaxation time $\trel$. For two integers $k,q\geq 1$, consider the graph constructed as follows:
\begin{enumerate}
\item Take a $3$-regular graph $\cE_k$ on $k$ vertices, satisfying the properties of Lemma~\ref{lem:lower-bound-random-graph};
\item Replace each node of $\cE_k$ by a clique of size $q$;
\item Replace each edge of $\cE_k$ by a path of length $q$.
\end{enumerate}
See Figure~\ref{fig:lower-bound-non-reg}.
\begin{figure}[ht]
\centering
\begin{tikzpicture}[scale=0.5,sibling distance=8em,
  every node/.style = {circle,minimum size=0.8mm,inner sep=0}]
  \tikzstyle{big node}=[circle,thick,draw=black,minimum size=6mm]
\node[big node] at (-2,1) (A) {$K_q$};
\node[big node] at (0,2) (B) {$K_q$};
\node[big node] at (2,0) (C) {$K_q$};
\node[big node] at (-5,1.5) (D) {$K_q$};
\node[big node] at (-3,5) (E) {$K_q$};
\node[big node] at (3,5) (F) {$K_q$};
\node[big node] at (5,1.5) (G) {$K_q$};
\node[big node] at (1,-3) (H) {$K_q$};
\draw (A)--(C) node[pos=0.33,fill] (ac1) {} node[pos=0.66,fill] (ac2) {};
\draw (A)--(H) node[pos=0.33,fill] (ah1) {} node[pos=0.66,fill] (ah2) {};
\draw (B)--(C) node[pos=0.33,fill] (bc1) {} node[pos=0.66,fill] (bc2) {};
\draw (A)--(D) node[pos=0.33,fill] (ad1) {} node[pos=0.66,fill] (ad2) {};
\draw (D)--(H) node[pos=0.33,fill] (dh1) {} node[pos=0.66,fill] (dh2) {};
\draw (H)--(G) node[pos=0.33,fill] (hg1) {} node[pos=0.66,fill] (hg2) {};
\draw (C)--(G) node[pos=0.33,fill] (cg1) {} node[pos=0.66,fill] (cg2) {};
\draw (F)--(G) node[pos=0.33,fill] (fg1) {} node[pos=0.66,fill] (fg2) {};
\draw (D)--(E)  node[pos=0.33,fill] (de1) {} node[pos=0.66,fill] (de2) {};
\draw (E)--(F) node[pos=0.33,fill] (ef1) {} node[pos=0.66,fill] (ef2) {};
\draw (B)--(E) node[pos=0.33,fill] (be1) {} node[pos=0.66,fill] (be2) {};
\draw (B)--(F) node[pos=0.33,fill] (bf1) {} node[pos=0.66,fill] (bf2) {};
\end{tikzpicture}
\caption{ }
\label{fig:lower-bound-non-reg}
\end{figure}
Such a graph has a number of vertices $n$ of order $kq$ and relaxation time $\trel$ of order $q^3$. Parameters $k$ and $q$ may then be tuned so as to obtained (almost) any possible $n$ and $\trel$. Now, to estimate correctly the number of edges, one needs to get the correct order for $k$. By Lemma~\ref{lem:lower-bound-random-graph}, a random walk on $\cE_k$ needs order $\sqrt{k}$ steps to make a cycle and thus be able to distinguish $\cE_k$ from an infinite $3$-regular tree. Since adding cliques and paths of size $q$ slows down the random walk by a factor of $q^3$ (the time to go from one clique to an other in the modified graph), the estimation of the number of edges on such a graph requires at least order $q^3\sqrt{k}\asymp \trel^{5/6}\sqrt{n}$ steps.

\subsection{Estimating the number of vertices on general graphs}
\label{subsec:vertices-general}

We first note that estimating the number of vertices might take much more time than estimating the number of edges. More precisely, we show that order $\trel^{5/6}\sqrt{n}$ steps may not be enough to estimate $n$. Indeed, consider the graph formed by a clique of size $\ell$ with path of length $q$ attached to each vertex of the clique, with $q\ll \ell$ (see Figure~\ref{fig:estim-vertices}).
\begin{figure}[ht]
\begin{tikzpicture}[scale=0.5,sibling distance=8em,
  every node/.style = {circle,minimum size=0.7mm,inner sep=0}]

\def \n {7}
\def \radius {2.5cm}
\node[draw, circle, minimum size=10mm] at (0,0) (A) {$K_\ell$};
\foreach \s in {1,...,\n}
{
  \node[draw, circle,fill] at ({360/\n * (\s - 1)}:\radius) {};
  \draw (A)-- ({360/\n * (\s - 1)}:\radius) node[pos=0.33,fill]  {} node[pos=0.66,fill]  {};

}
\end{tikzpicture}
\caption{ }
\label{fig:estim-vertices}
\end{figure}
The number $n$ of vertices is of order $\ell q$, and, as $q\ll \ell$, the number $m$ of edges if of order $\ell^2$. Moreover, the relaxation time is of order $q^2$ (this can be seen by a coupling argument). Estimating $m$ is relatively easy: starting from the end of one path, the walk has to traverse it to reach the clique, which takes time $q^2$, and then to wait for a collision in the clique, which, by the birthday paradox, takes time $\sqrt{\ell}$. Estimating $n$ however takes more time: starting from the clique, the walk has to visit a positive fraction of at least one of the paths, and this takes time $\ell q$. As soon as $q\ll \ell^{3/7}$, we have $\ell q\gg \sqrt{\ell}q^{13/6} \asymp \trel^{5/6}\sqrt{n}$.

\medskip

Estimating $n$ might thus require more time. However, once a good estimate for the number of edges is known, it is quite easy to deduce an estimate for the number of vertices. Indeed, what remains to estimate is just the mean degree. Consider the function $f:x\in V\mapsto f(x)=\frac{1}{\deg(x)}$, and note that $\EE_\pi f=\frac{n}{2m}$. Applying \citep[Proposition 12.19]{levin2009markov} to the function $f$, we know that for $r\geq \tmix(\varepsilon/2)$ and $t\geq \frac{16\var_\pi f}{\varepsilon(\EE_\pi f)^2}\trel$, for all $x\in V$,
\begin{eqnarray*}
\PP_x\left(\left| \frac{1}{t}\sum_{s=0}^{t-1}f(X_{r+s})-\EE_\pi f\right|>\frac{\EE_\pi f}{2}\right)&\leq & \varepsilon\, .
\end{eqnarray*}
Observing that $\var_\pi f\leq \EE_\pi f^2=(2m)^{-1}\sum_{u}\deg(u)^{-1}$ and that $\tmix(\varepsilon/2)\lesssim \log(1/\varepsilon)\tunif$, the mean degree can be estimated with error probability less than $\varepsilon$ in time of order
$$
\log(1/\varepsilon)\tunif + \frac{\trel m}{\varepsilon n^2}\sum_{u\in V}\deg(u)^{-1}\lesssim \varepsilon^{-1}\tunif \frac{m}{n}\, \cdot
$$
Note that this is optimal by the previous example of Figure~\ref{fig:estim-vertices}, for which $\ell q \asymp \tunif m/n$. Altogether, the number of vertices of a connected graph can be estimated by random walks in time
$$
\varepsilon^{-1}\left(\left(\trel^{3/4}\sqrt{m/d}\right)\wedge\left(\trel^{5/6}\sqrt{n}\right) +\tunif \frac{m}{n}\right)\, .
$$

\section{No self-stopping algorithms in general}\label{sec:no-self-stop}

In this section, we show that one can not hope for a general sublinear self-stopping algorithm, even when restricting to graphs with  polylog mixing time.

Let $\cC$ be the class of graphs $G$ such that $\tunif(G)\leq (\log n_G)^3$.

Consider the following process on a graph, called \emph{random walk with restarts}: at each time step $t\geq 0$, based on $\Phi(X_0,\dots,X_t)$, the process decides whether it wants to make a random walk step from $X_t$, or to reset back to the starting point $x$. A self-stopping algorithm is based on the profile of a random walk with restarts, up to some stopping time $\tau$. More precisely, it relies on a function $\stopp:\cS\to \{0,1\}$. Defining
$$
\tau=\inf\left\{t\geq 0,\, \stopp\left(\Phi(X_0^t)\right)=1\right\}\, ,
$$
where $X_0^t=(X_0,\dots,X_t)$ is the trajectory of a random walk with restarts up to time $t$, then the self-stopping algorithm defined by \stopp\ and \est\ returns the value $\est\left(\Phi(X_0^\tau)\right)$.

\begin{proposition}
\label{prop:no-self-stop}
There exists $\delta>0$, such that, for all functions $\stopp$ and $\est$, there is an infinite sequence of graphs $G\in\cC$ and $x\in V(G)$ such that
\begin{eqnarray*}
\PP^G_x\left(\{\tau\geq \delta n_G\}\cup\left\{\Big|\frac{\est\left(\Phi(X_0^\tau)\right)}{n_G}-1\Big| >\frac{1}{2}\right\}\right)&\geq &\frac{1}{4}\, ,
\end{eqnarray*}
where $X$ is a \rw\ with restarts and $\tau=\inf\{t\geq 0,\, \stopp\left(\Phi(X_0^\tau)\right)=1\}$.
\end{proposition}

\begin{proof}[Proof of Proposition~\ref{prop:no-self-stop}]
Consider a $3$-regular expander $G$ on $n$ vertices and a graph $G^\star$ obtained from $G$ as follows: let $G^{(1)},\dots,G^{(2^n)}$ be $2^n$ identical copies of $G$. For all $i\in\{1,\dots,2^n\}$, choose three distinct vertices $(u_i,v_i,w_i)$ uniformly at random in $V(G^{(i)})$. Now let $F$ be some other $3$-regular expander on $2^n$ vertices, labelled from $1$ to $2^n$. For all $1\leq i\leq 2^n$, if $i$ has neighbors $j<k<\ell$ in $F$, put an edge between $u_i$ and $u_j$, between $v_i$ and $v_k$, between $w_i$ and $w_\ell$. Let $G^\star$ be the resulting graph (on $|V(G^\star)|=n2^n$ vertices). Note that, as $F$ is an expander, and as the random walk on $G^\star$ needs order $n$ steps to go from some $u_i$ to either $v_i$ or $w_i$, we have $\tunif(G^\star)\lesssim n\log (2^n)$, so that both $G$ and $G^\star$ belong to the class $\cC$. It is not hard to check that one can find $y\in V(G^{(1)})$ and $\delta>0$, such that
\[
\PP^{G^\star}_y\left(\bigcap_{s=0}^{\delta n}\left\{Y_s\not\in\{u_1,v_1,w_1\}\right\}\right)\geq \frac{2}{3}\, \cdot
\]
Therefore, there exist starting points $(x,y)\in V(G)\times V(G^\star)$, and a coupling $(X,Y)$ of random walks with restarts at $x$ and $y$ (for the same restarting rule) such that
\begin{equation}
\label{eq:coupling-GH}
\PP_{x,y}\left(\bA_t\right) \geq  \frac{2}{3}\, ,\quad\mbox{with}\quad \bA_t=\left\{\Phi(X_0^t)=\Phi(Y_0^t)\right\}\quad\mbox{and}\quad t=\delta n\, .
\end{equation}
Let $\est:\cS\to \N$ be an estimator and $\stopp:\cS\to\{0,1\}$. For $(Z,H)\in\left\{ (X,G), (Y,G^\star)\right\}$, define
\[
B_H^{Z}=\left\{\tau^Z<\delta \left|V(H)\right|\right\}\cap\left\{\left| \frac{\est\left(\Phi(Z_0^{\tau^Z})\right)}{|V(H)|}-1\right|\leq \frac{1}{2}\right\}\, ,
\]
where $\tau^Z=\inf\left\{s\geq 0,\, \stopp\left(\Phi(Z_0^s)\right)=1\right\}$. Assume that we both have $\PP_{x}\left(B_G^{X}\right)\geq 3/4$ and $\PP_{y}\left(B_{G^\star}^{Y}\right)\geq 3/4$. Then, by~\eqref{eq:coupling-GH},
\begin{displaymath}
\PP_{x,y}\left(B_G^{X}\given \bA_t\right)=\frac{\PP_{x,y}\left(B_G^{X}\cap \bA_t\right)}{\PP_{x,y}(\bA_t)}\geq 1-\frac{1-\PP_{x}\left(B_G^{X}\right)}{\PP_{x,y}(\bA_t)}\geq \frac{5}{8}\, ,
\end{displaymath}
and similarly, $\PP_{x,y}\left(B_{G^\star}^{Y}\given \bA_t\right) \geq \frac{5}{8}$, so that $\PP_{x,y}\left(B_G^{X}\cap B_{G^\star}^{Y}\given \bA_t\right) \geq  \frac{1}{4}$. However, on the event $\bA_t$, we have $\{\tau^X<\delta |V(G)|\}\cap\{\tau^Y<\delta |V(G^\star)|\}=\{\tau^X<\delta n\}\cap\{\tau^Y=\tau^X\}$, so that $\est\left(\Phi(X_0^{\tau^X})\right)=\est\left(\Phi(Y_0^{\tau^Y})\right)$ and the events $B_G^{X}$ and $B_{G^\star}^{Y}$ can not occur simultaneously, implying a contradiction.
\end{proof}

\section{A self-stopping algorithm for the number of edges}\label{sec:algo-edges}

Let $G=(V,E)$ be a finite connected graph and let $\tau$ be an upper-bound on the relaxation time $\trel$.

\begin{algo}
\label{algo:edges}
For $q=0,1,\dots$, iterate the following procedure until stopped:
\begin{itemize}
\item let $\hat{m}=2^q$ be the current guess for the number of edges and let $t=t_{q}= \tau^{3/4}\sqrt{2\hat{m}}$.
\item let $R=R_q=\lceil 8\log(4/\varepsilon) + 16\log (q+1) \rceil$ and repeat the following experiment $R$ times.
\begin{itemize}
\item let $X^{(1)},Y^{(1)},\dots,X^{(K)},Y^{(K)}$ be $2K$ independent \lrw\ started from $x$ (for a fixed integer $K\geq 1$ to be specified later) and define
\[
\cQ_t = \frac{1}{K}\sum_{\ell=1}^K \cI_t^{(\ell)}\, ,\quad\mbox{where }\quad
\cI_t^{(\ell)} =  \sum_{i,j=0}^{t-1} \frac{1}{\deg(X_i^{(\ell)})}\ind_{\left\{X_i^{(\ell)}=Y_j^{(\ell)}\right\}}\, .
\]
\item If $\cQ_t\geq 18\tau^{3/2}$, call the experiment a success.
\end{itemize}
\item If the number of successes is larger than $R/2$, then stop and estimate $m$ by $\hat{m}=2^q$; otherwise, go from $q$ to $q+1$.
\end{itemize}
\end{algo}

\begin{proposition}
\label{prop:algo-edges}
Algorihtm~\ref{algo:edges} satisfies the two following properties:
\begin{enumerate}
\item The probability that the algorithm stops at a value of $q$ such that $2^q <m$ or $2^q>38m$ is smaller than $\varepsilon$.
\item The expected running time of the algorithm is $O\left(\sqrt{m}\tau^{3/4}\log\log m\right)$.
\end{enumerate}
\end{proposition}

\begin{proof}[Proof of Proposition \ref{prop:algo-edges}]
$(1)$  By equation~\eqref{eq:prop1-It-exp} in Proposition~\ref{prop:intersections} and since $d\geq 1$, it always holds that
\begin{equation}
\label{eq:exp-Q}
\frac{\tau^{3/2}\hat{m}}{m} \leq \EE_x \cQ_t  \leq  \frac{\tau^{3/2}\hat{m}}{m}+ 16\tau^{3/2}\, .
\end{equation}
Assume that $q$ is such that $\hat{m}=2^q < m$. Then the expectation of $\cQ_t$ is smaller than $17\tau^{3/2}$. By Chebyshev's Inequality,
$$
\PP_x\left(\cQ_t\geq 18\tau^{3/2}\right) \leq  \PP_x\left(\cQ_t-\EE_x\cQ_t\geq \tau^{3/2}\right) \lesssim  \frac{\var_x \cI_t}{K \tau^3}\, .
$$
Now by equation~\ref{eq:prop1-It-var} and since $t<\tau^{3/4}\sqrt{2m}$, we have $\var_x \cI_t \lesssim \tau^3$. Hence, we may choose $K$ large enough such that $\PP_x\left(\cQ_t\geq 20\tau^{3/2}\right)\leq 1/4$. Using Hoeffding's Inequality, the probability that there are more than $R/2$ successes at this step is smaller than $\exp\left(-R/8\right)=\frac{\varepsilon}{4} (q+1)^{-2}$. Taking a union bound, the probability for the algorithm to stop at a value of $q$ such that $2^q < m$ is smaller than $\varepsilon/2$.

Let now $q$ be such that $\hat{m}=2^q>19m$. By equation~\eqref{eq:exp-Q}, the expectation of $\cQ_t$ is larger than $19\tau^{3/2}$. Hence
$$
\PP_x\left(\cQ_t < 18 \tau^{3/2}\right)\leq  \PP_x\left(\cQ_t <\frac{18}{19}\EE_x\cQ_t \right) \lesssim  \frac{\var_x \cI_t}{K \left(\EE_x \cI_t\right)^2}\, \cdot
$$
Again, equation~\ref{eq:prop1-It-var} entails that the constant $K$ may be chosen such that the above probability is smaller than $1/4$. And by Hoeffding's Inequality, the probability that there are less than $R/2$ successes is smaller than $\exp(-R/8)\leq \varepsilon/4$. Clearly, the probability to stop at a step $q$ with $2^q> 38m$ is smaller than the probability not to have stopped at $q^\star=\inf\{q\geq 0,\; 2^q > 19 m\}$, which is smaller than $\varepsilon/4$.

$(2)$  By the above, for all $q> q^\star$, the probability that the algorithm stops at step $q$ is smaller than $(\varepsilon/4)^{q-q^\star}$. Now the running time up to step $q$ is smaller, up to constant factors, than  $\displaystyle{\sum_{i=0}^q R_i t_i\lesssim R_q t_q}$, so that the expected running time is smaller, up to constant factors, than
\begin{eqnarray*}
R_{q^\star}t_{q^\star}+\sum_{q>q^\star} \left(\frac{\varepsilon}{4}\right)^{q-q^\star}R_q t_q&=&O\left(\sqrt{m}\tau^{3/4}\log\log m\right)\, .
\end{eqnarray*}
\end{proof}

\begin{remark}
If the graph is $d$-regular of if the minimum degree $d$ is known, Proposition~\ref{prop:intersections} allows to design an algorithm which estimates $m$ (or rather $n$ in the case of regular graphs) in expected time $O\left(\sqrt{m/d}\tau^{3/4}\log\log (m/d)\right)$.
\end{remark}

\section{Algorithms for the mixing time}\label{sec:algo-mixing}

The number of intersections may also be used to estimate the mixing time from a given vertex $x\in V$. Assume that the number of edges $m$ in $G=(V,E)$ is known. Let
\begin{eqnarray*}
d_x(t)&=&\sqrt{\sum_y \pi(y)\left(\frac{\PP_x(X_t=y)}{\pi(y)}-1\right)^2}
\end{eqnarray*}
be the $\ell_2(\pi)$-distance between $\PP_x(X_t\in \cdot)/\pi(\cdot)$ and $1$. Our goal now is to estimate
\begin{eqnarray*}
t_x(\delta)&=&\inf\left\{t\geq 0,\; d_x(t)^2 \leq\delta\right\}\, , \quad 0<\delta<1\, .
\end{eqnarray*}
Before describing a self-stopping algorithm to estimate $t_x(\delta)$, we prove the following useful lemma.

\begin{lemma}
\label{lem:var-cov}
Let $X,Y,Z$ be three independent random walks started at $x$ and let $\cL_t^{(X,Y)}=\cI_{2t}^{(X,Y)}-\cI_t^{(X,Y)}$ be the weighted number of intersections of $X$ and $Y$ between $t$ and $2t$. Define $\cL_t^{(X,Z)}$ similarly. For all $t\geq 0$,
\begin{eqnarray}
\label{eq:cL-dist}
\EE_x\cL_t^{(X,Y)}&=& \sum_{i,j=t}^{2t-1} \frac{d_x\left(\frac{i+j}{2}\right)^2+1}{2m}\, ,
\end{eqnarray}
\begin{eqnarray}
\label{eq:var_cL}
\var_x \cL_t^{(X,Y)} &\lesssim & \EE_x \cL_t^{(X,Y)} \max_u \EE_u\cI_t\, ,
\end{eqnarray}
and
\begin{eqnarray}
\label{eq:cov_cL}
\cov_x\left(\cL_t^{(X,Y)},\cL_t^{(X,Z)}\right)&\lesssim & \left(\EE_x \cL_t^{(X,Y)}\right)^{3/2}\sqrt{ \max_u \EE_u\cI_t}\, .
\end{eqnarray}
\end{lemma}

\begin{proof}[Proof of Lemma~\ref{lem:var-cov}]
By reversibility, $d_x(t)^2=\frac{\PP_x(X_{2t}=x)}{\pi(x)}-1$, and
$$
\EE_x\cL_t^{(X,Y)} =\frac{1}{\deg(x)}\sum_{i,j=t}^{2t-1} \PP_x(X_{i+j}=x)=\sum_{i,j=t}^{2t-1} \frac{d_x\left(\frac{i+j}{2}\right)^2+1}{2m}\, \cdot
$$
Moving on to~\eqref{eq:var_cL}, defining $\G_{t\rightarrow 2t}(x,u)=\G_{2t}(x,u)-\G_{t}(x,u)$, one easily checks that
$$
\EE_x \cL_t^{(X,Y)}= \sum_u \frac{\G_{t\rightarrow 2t}(x,u)^2}{\deg(u)}\, ,
$$
and that
$$
\EE_x\left(\left(\cL_t^{(X,Y)}\right)^2\right)\lesssim  \sum_{u,v}\frac{\G_{t\rightarrow 2t}(x,u)^2\G_{t}(u,v)^2}{\deg(u)\deg(v)} =  \sum_{u}\frac{\G_{t\rightarrow 2t}(x,u)^2}{\deg(u)}\EE_u\cI_{t} \, ,
$$
which implies
$$
\EE_x\left(\left(\cL_t^{(X,Y)}\right)^2\right)\lesssim  \EE_x \cL_t^{(X,Y)} \max_u \EE_u\cI_t\, .
$$
Finally, to establish~\eqref{eq:cov_cL}, observe that
$$
\EE_x\left(\cL_t^{(X,Y)}\cL_t^{(X,Z)}\right) \lesssim  \sum_{u,v} \frac{\G_{t\rightarrow 2t}(x,u)^2\G_{t}(u,v)\G_{t\rightarrow 2t}(x,v)}{\deg(u)\deg(v)}\, ,
$$
and that, by Cauchy-Schwarz Inequality,
$$
\EE_x\left(\cL_t^{(X,Y)}\cL_t^{(X,Z)}\right)
\leq \left( \EE_x \cL_t^{(X,Y)}\right)^{3/2}\sqrt{ \max_u \EE_u\cI_t}\, .
$$
\end{proof}

\begin{algo}
\label{algo:mixing}
For $q=0,1,\dots$, iterate the following procedure until stopped:
\begin{itemize}
\item Let $t=t_q=2^q$ be the current guess for the mixing time $t_x(\delta)$ and let \linebreak $K= K_q = \left\lceil C\delta^{-2} \left\lceil\sqrt{m} t^{-1/4}\right\rceil\right\rceil$, for a constant $C>0$ to be specified later.
\item Let $R=R_q=\lceil 8\log(4/\varepsilon) + 16\log (q+1) \rceil$ and repeat the following experiment $R$ times.
\begin{itemize}
\item Let $X^{(1)},\dots,X^{(K)}$ be $K$ independent \lrw\ started from $x$ and define
\[
\cL_t = \dfrac{1}{{K\choose 2}}\sum_{1\leq \ell<k\leq K} \cL_t^{(\ell,k)}\, ,\quad\mbox{where }\quad
\cL_t^{(\ell,k)} = \sum_{i,j=t}^{2t-1} \frac{1}{\deg(X_i^{(\ell)})}\ind_{\left\{X_i^{(\ell)}=X_j^{(k)}\right\}}\, .
\]
\item If $\cL_t\leq \left(1+\frac{\delta}{2}\right)\frac{t^2}{2m}$, call the experiment a success.
\end{itemize}
\item If the number of successes is larger than $R/2$, then stop and estimate $t_x(\delta)$ by $t=2^q$; otherwise, go from $q$ to $q+1$.
\end{itemize}
\end{algo}

\begin{proposition}
\label{prop:algo-mixing}
Algorithm~\ref{algo:mixing} satisfies the two following properties:
\begin{enumerate}
\item The probability that the algorithm stops at a value of $q$ such that $2^q < t_x(\delta)/2$ or $2^q>2t_x(\delta/4)$ is smaller than $\varepsilon$.
\item The expected running time of the algorithm is $O\left(\delta^{-2}\sqrt{m}t_x(\delta/4)^{3/4}\log\log t_x(\delta/4)\right)$.
\end{enumerate}
\end{proposition}

\begin{proof}[Proof of Proposition~\ref{prop:algo-mixing}]
$(1)$  Assume that $q$ is such that $t=2^q < t_x(\delta)/2$. By equation~\eqref{eq:cL-dist}, the expectation of $\cL_t$ is larger than $(1+\delta)\frac{t^2}{2m}$. By Chebyshev's Inequality,
\begin{equation}\label{eq:prob-Lt}
\PP_x\left(\cL_t\leq \left(1+\frac{\delta}{2}\right)\frac{t^2}{2m}\right) \lesssim  \frac{\var_x \cL_t}{\delta^2(\EE_x\cL_t )^2}\, \cdot
\end{equation}
Since for all $\ell,\ell',k,k'$ pairwise distinct, $\cov_x\left(\cL_t^{(\ell,k)}, \cL_t^{(\ell',k')}\right)=0$, we have
$$
\var_x \cL_t \lesssim  \frac{\var_x \cL_t^{(X,Y)}}{K^2} + \frac{\cov_x\left(\cL_t^{(X,Y)},\cL_t^{(X,Z)}\right)}{K} \, ,
$$
so that, by Lemma~\ref{lem:var-cov} and using that $\EE_x\cL_t\geq t^2/2m$, we get
$$
 \frac{\var_x \cL_t}{\delta^2(\EE_x\cL_t )^2} \lesssim \kappa +\sqrt{\kappa}\, ,
$$
where
$$
\kappa = \frac{\max_u \EE_u\cI_t}{C^2\left\lceil \sqrt{m}/t^{1/4}\right\rceil^2(t^2/m)}\, .
$$
Now, if $t\leq \frac{36m^2}{d}$, then applying Lemma~\ref{lem:bound-green} directly in~\eqref{eq:intersection-return} yields $\max_u \EE_u\cI_t \lesssim t^{3/2}$. On the other hand, if $t>\frac{36m^2}{d}$, then by~\eqref{eq:bound-trel}, $t\gtrsim\trel^{3/4}\sqrt{m/d}$, which by Proposition~\ref{prop:intersections} yields $\max_u \EE_u\cI_t\lesssim t^2/m$. Hence, in both cases, we have $\kappa \lesssim 1/C^2$, and the constant $C$  can be made large enough so that the right-hand side in~\eqref{eq:prob-Lt} is smaller than $1/4$. Using Hoeffding's Inequality, the probability that there are more than $R/2$ successes is then smaller than $\exp\left(-R/8\right)=\frac{\varepsilon}{4} (q+1)^{-2}$. Taking a union bound, we obtain that the probability for the algorithm to return an estimate smaller than $t_x(\delta)/2$ is smaller than $\varepsilon/2$.

Now let $q$ be such that $t=2^q > t_x(\delta/4)$. Then $\EE_x\cL_t\leq (1+\delta/4)\frac{t^2}{2m}$ and by Chebyshev's Inequality
\begin{eqnarray*}
\PP_x\left(\cL_t > \left(1+\frac{\delta}{2}\right)\frac{t^2}{2m}\right)&\lesssim & \frac{\var_x \cL_t}{\delta^2(t^2/m)^2}\, .
\end{eqnarray*}
By the same arguments as above, the constant $C$ can be chosen large enough so that the above probability is smaller than $1/4$. By Hoeffding's Inequality, the probability that there are less than $R/2$ successes is smaller than $\exp\left(-R/8\right)\frac{\varepsilon}{4}$. Clearly, the probability to stop at a value $q$ such that $2^q> 2t_x(\delta/4)$ is smaller than the probability not to have stopped at $q^\star=\inf\{q\geq 0,\; 2^q > t_x(\delta/4)\}$, which is smaller than $\varepsilon/4$.

$(2)$  By the above, for all $q>q^\star$, the probability that the algorithm stops at step $q$ is smaller than $(\varepsilon/4)^{q-q^\star}$. Moreover, the running time up to step $q$ is smaller, up to constant factors, than $\displaystyle{\sum_{i=0}^q R_i K_i t_i \lesssim \delta^{-2} \sqrt{m}(t_q)^{3/4}\log (q+1)}$. Altogether, the expected running time is less, up to constant factor, than
\begin{eqnarray*}
\frac{\sqrt{m}}{\delta^2}(t_{q^\star})^{3/4}\log (q^\star+1) +\sum_{q>q^\star}(1/4)^{q-q^\star}\frac{\sqrt{m}}{\delta^2}(t_q)^{3/4}\log (q+1)\, ,
\end{eqnarray*}
which is $O\left(\delta^{-2}\sqrt{m}t_x(\delta/4)^{3/4}\log\log t_x(\delta/4)\right)$.
\end{proof}

\begin{remark}
If the graph is $d$-regular of if the minimum degree $d$ is known, Proposition~\ref{prop:intersections} actually allows to design an algorithm which estimates $t_x(\delta)$  in expected time $O\left(\delta^{-2}\sqrt{m/d}t_x(\delta/4)^{3/4}\log\log t_x(\delta/4)\right)$.
\end{remark}

We assume, for simplicity, that the true value of $m$ is known. However, our estimation scheme can easily be extended to the case where only a good approximation of $m$ is available. Combining Proposition~\ref{prop:algo-edges} and ~\ref{prop:algo-mixing} then entails the following corollary.

\begin{corollary}
An upper-bound $\tau$ on the uniform mixing time can be used to precisely estimate both the number of edges and the mixing time from $x$, via a self-stopping algorithm with time complexity $O\left(\sqrt{m}\tau^{3/4}\log\log m\right)$.
\end{corollary}

\noindent{\bf Acknowledgement.} The question of estimating the mixing time with random walks trajectories was posed by G\'abor Lugosi, during the \emph{Eleventh annual workshop in Probability and Combinatorics}, Barbados, April 1-8, 2016. We are grateful to him for bringing this problem to our attention, and we thank the Bellairs Institute where this work was initiated. RO was funded by a {\em Bolsa de Produtividade em Pesquisa} from CNPq, Brazil and a {\em Cientista do Nosso Estado} grant from FAPERJ, Rio de Janeiro, Brazil. His work in this article is part of the activities of FAPESP Center for Neuromathematics (grant \# 2013/ 07699-0 , FAPESP - S.Paulo Research Foundation).

\bibliographystyle{abbrvnat}
\bibliography{biblio}

\end{document}